\documentclass[10pt,a4paper,twoside]{article}
\usepackage{amsmath,amssymb,theorem}
\usepackage{color}

\newcommand{\assign}{:=}
\newcommand{\tmem}[1]{{\em #1\/}}
\newcommand{\tmop}[1]{\ensuremath{\operatorname{#1}}}
\newcommand{\tmtextbf}[1]{{\bfseries{#1}}}
\newcommand{\tmtextit}[1]{{\itshape{#1}}}
\newcommand{\tmtextsc}[1]{{\scshape{#1}}}
\newenvironment{proof}{\noindent\textbf{Proof\ }}{\hspace*{\fill}$\Box$\medskip}
{\theorembodyfont{\rmfamily}\newtheorem{example}{Example}}
\newtheorem{corollary}{Corollary}
\newtheorem{lemma}{Lemma}
\newtheorem{proposition}{Proposition}
{\theorembodyfont{\rmfamily}\newtheorem{remark}{Remark}}
\newtheorem{theorem}{Theorem}

\begin{document}

\title{Energy decay for solutions of the wave equation with general memory
boundary conditions.}

\author{Pierre Cornilleau\\
Ecole Centrale de Lyon\\
Institut Camille Jordan, UMR CNRS 5208\\
36 avenue Guy de Collongue\\
69134 Ecully Cedex France\\
pcornill@ec-lyon.fr
\\
\\
Serge Nicaise\\
Universit\'e de Valenciennes et du Hainaut Cambr\'esis\\
LAMAV,  FR CNRS 2956, \\
Institut des Sciences et Techniques de Valenciennes\\
F-59313 - Valenciennes Cedex 9 France\\
Serge.Nicaise@univ-valenciennes.fr\\
}

\maketitle

\begin{abstract}
We consider the wave equation in a smooth domain subject to
Dirichlet boundary conditions on one part of the boundary and
dissipative boundary conditions of memory-delay type on the
remainder part of the boundary, where a general borelian measure is
involved. Under quite weak assumptions on this measure,  using
  the multiplier method and a standard integral inequality we show
    the exponential stability of the system.
    Some examples of measures   satisfying
our hypotheses are given,   recovering  and extending some of the
results from the literature.
\end{abstract}

\section*{Introduction}

We consider the wave equation subject to Dirichlet boundary
conditions on one part of the boundary and dissipative boundary
conditions of memory-delay type on the remainder part of the
boundary. More precisely, let $\Omega$ be a bounded open connected
set of $\mathbb{R}^n (n \geq 2)$ such that, in the sense \ of
Ne\v{c}as (\cite{Ne}), its boundary $\partial \Omega$ is of class
$\mathcal{C}^2$. Throughout the paper, $I$   denotes the $n \times
n$ identity matrix, while $A^s$   denotes the symmetric part of a
matrix $A$. Let $m$ be a $\mathcal{C}^1$ vector field on
$\bar{\Omega}$ such that
\begin{equation}
  \label{m} \inf_{\bar{\Omega}} \tmop{div} (m) > \sup_{\bar{\Omega}}
  (\tmop{div} (m) - 2 \lambda_m)
\end{equation}
where $\lambda_m (x)$ is the smallest eigenvalue function of the
real symmetric matrix $\nabla m (x)^s$.

\begin{remark}
  The set of all $\mathcal{C}^1$ vector fields on $\bar{\Omega}$ such that
(\ref{m})  holds is an open cone. If $m$ is in this set, we denote
  \[ c (m) = \frac{1}{2} \left( \inf_{\bar{\Omega}} \tmop{div} (m) -
     \sup_{\bar{\Omega}} (\tmop{div} (m) - 2 \lambda_m) \right) . \]
\end{remark}

\begin{example}
  \begin{itemize}
    \item An affine example is given by
    \[ m (x) = (A_1 + A_2) (x - x_0), \]
    where $A_1$ is a definite positive matrix, $A_2$ a skew-symmetric matrix
    and $x_0$ any point in $\mathbb{R}^n$.

    \item A non linear example is
    \[ m (x) = (dI + A) (x - x_0) + F (x) \]
    where $d > 0$, $A$ is a skew-symmetric matrix, $x_0$ any point in
    $\mathbb{R}^n$ and $F$ is a $\mathcal{C}^1$ vector field on $\bar{\Omega}$
    such that
    \[ \sup_{x \in \bar{\Omega}} \|(\nabla F (x))^s \|< \frac{d}{n} \]
    ( $\| \cdot \|$ stands for the usual $2$-norm of matrices).
  \end{itemize}
\end{example}

We define a partition of $\partial \Omega$ in the following way.
Denoting by $\nu (x)$ the normal unit vector pointing outward of
$\Omega$ at a point \ $x \in \partial \Omega$ , we consider a
partition $(\partial \Omega_N, \partial \Omega_D)$ of the boundary
such that the measure of $\partial\Omega_D$ is positive and that
\begin{equation}\label{sign}
 \partial \Omega_N \subset \{x \in \partial \Omega, \text{} m (x) \cdot \nu (x)
   \geqslant 0\}, \partial \Omega_D \subset \{x \in \partial \Omega,  m
   (x) \cdot \nu (x) \leqslant 0\}.
\end{equation}

Furthermore, we assume
\begin{equation}\label{disjoint}
 \overline{\partial \Omega_D} \cap \overline{\partial \Omega_N} =
   \varnothing
   \text{ or } \ m\cdot n \leq 0 \text{ on } \overline{\partial \Omega_D} \cap \overline{\partial \Omega_N}
\end{equation}
where $n$ stands for the normal unit vector pointing outward of $\partial \Omega_N $ when considering $\partial \Omega_N $ as a sub-manifold of $\partial \Omega $.

On this domain, we consider the following delayed wave problem:
\[ (S) \left\{ \begin{array}{l}
     u'' - \Delta u = 0\\
     u = 0\\
     \partial_{\nu} u + m\cdot \nu \left( \mu_0 u' (t) + \int_0^t u' (t - s) d \mu
     (s) \right) = 0\\
     u (0) = u_0\\
     u' (0) = u_1
   \end{array} \right. \left. \begin{array}{l}
     \text{in } \mathbb{R}_+^{\ast} \times \Omega \hspace{0.25em},\\
     \text{on } \mathbb{R}_+^{\ast} \times \partial \Omega_D
     \hspace{0.25em},\\
     \text{on } \mathbb{R}_+^{\ast} \times \partial \Omega_N
     \hspace{0.25em},\\
     \text{in } \Omega \hspace{0.25em},\\
     \text{in } \Omega \hspace{0.25em},
   \end{array} \right. \]
where $u'$ (resp. $u^{^{\prime \prime}}$) is the first (resp.
second) time-derivative of $u$, $\partial_{\nu} u = \nabla u\cdot
\nu$ is the normal outward derivative of $u$ on $\partial \Omega$.
Moreover $\mu_0$ is some positive constant and $\mu$ is a borelian
measure on $\mathbb{R}^+$.

The above problem covers the case of a problem with memory type as
studied for instance in \cite{ACS,CG,guesmia:99,NPaa}, when the
measure $\mu$ is given by
\begin{equation}\label{noyauconv}
d\mu(s)= k(s) ds, \end{equation} where $ds$ stands for the Lebesgue
measure and $k$ is non negative kernel. But it also covers the case
of a problem with a delay  as studied for instance in
\cite{NP,NP2,NicValein}, when the measure $\mu$ is given by
\begin{equation}\label{diracdelay}
 \mu= \mu_1 \delta_{\tau},\end{equation}
 where $\mu_1$ is a non negative constant and $\tau>0$ represents
 the delay.
 An intermediate case treated in \cite{NP2} is the case when
\begin{equation}\label{intermediaire}
d\mu(s)= k(s) \chi_{[\tau_1,\tau_2]}(s) ds, \end{equation} where
$0\le \tau_1 <\tau_2,$   $\chi_{[\tau_1,\tau_2]}$ is the
characteristic equation of the interval $[\tau_1,\tau_2]$  and $k$
is a non negative function in $L^\infty([\tau_1,\tau_2])$.

 A closer look at the decay results obtained in these references
shows that there are different ways to quantify the energy of $(S)$.
More precisely for the measure of the form (\ref{noyauconv}),
 the exponential or polynomial decay of an appropriated energy is proved in \cite{ACS,CG,guesmia:99,NPaa},
 by combining the multiplier method (or differential geometry arguments) with the use of suitable
 Lyapounov functionals (or integral inequalities) under the
 assumptions that the kernel $k$ is sufficiently smooth and has a
 certain decay at infinity.
 On the other hand for a measure of delay type like (\ref{diracdelay})
 or (\ref{intermediaire}), the exponential stability of the system
 was proved in \cite{NP,NP2,NicValein} by proving an observability
 estimate obtained by assuming that the term
 $\int_0^t u' (t - s) d \mu(s)$ is sufficiently small with respect to $\mu_0 u'
 (t)$.
Consequently, our goal is here to obtain some uniform decay results
 in the general context described above with a similar assumptions than in \cite{NP,NP2,NicValein}.
 More precisely, we will show in this paper that if there
exists $\alpha>0$ such that
\begin{equation}\label{mu0}
 \mu_{\text{tot}}\assign  \int_0^{+\infty} e^{\alpha s} d|\mu|(s)< \mu_0
\end{equation}
where $|\mu|$ is the absolute value of the measure $\mu$, then the
above problem $(S)$ is exponentially stable.

The paper is organized as follows: in the first two sections, we
explain how to define an energy using some basic measure theory.
Using well-known results, we obtain the existence of energy
solutions. In this setting we present and prove our stabilization
result in the third section. Examples of measures $\mu$ satisfying
our hypotheses are given in the end of the paper, where we show that
we recover and extend some of the results from the references cited
above.

Finally in the whole paper we use the notation  $A\lesssim B$ for
the estimate $A\leq C B$   with some constant $C$ that only depends
on $\Omega$, $m$ or $\mu$.

\section{First results}

In this section we show that the assumption (\ref{mu0}) implies the
existence of some borelian finite measure $\lambda$ such that
\begin{equation}\label{surmesure}
\lambda(\mathbb{R}^+)<\mu_0,\quad  |\mu|\leq \lambda
\end{equation}
(in the sense that, for every measurable set $\mathcal{B}$,
$|\mu|(\mathcal{B}) \leq \lambda(\mathcal{B})$) and
\begin{equation}\label{integration}
\text{for all   measurable set }  \mathcal{B}, \int_\mathcal{B}
\lambda([s,+\infty))ds \leq \alpha^{-1} \lambda(\mathcal{B})
\end{equation}

Indeed  we  show the following equivalence:
\begin{proposition}\label{measure}
Let $\mu$ be a borelian positive measure on $\mathbb{R}^{+}$ and $\mu_0$ some positive constant. The following properties are equivalent:
\begin{itemize}
\item $\exists \alpha>0$ such that
\[\int_0^{+\infty} e^{\alpha s}d\mu(s)<\mu_0.\]
\item There exists a borelian measure $\lambda$ on $\mathbb{R}^{+}$ such that
\[\lambda(\mathbb{R}^{+})<\mu_0,\quad \mu \leq \lambda  \]
and, for some constant $\beta>0$,
\[\text{for all   measurable set }  \mathcal{B}, \int_\mathcal{B} \lambda([s,+\infty))ds \leq \beta^{-1} \lambda(\mathcal{B}).\]
\end{itemize}
\end{proposition}
\begin{proof}
We introduce the application $T$ from the set of positive
 borelian measures into itself as follows: if $\mu$ is some positive
 borelian measure, we define a positive borelian measure
$T(\mu)$ by
$$T(\mu)(\mathcal{B})=\int_{\mathcal{B}}\mu([s,+\infty))ds,$$
if $\mathcal{B}$ is any measurable set.
\par $(\Leftarrow)$ If $\lambda$ fulfills the second property, then it immediately follows that
\[\forall n\in \mathbb{N}, \beta^n T^n(\mu)\leq \lambda,\]
where as usual $T^n$ is the composition $T\circ T \cdots \circ T$
$n$-times.
 A summation consequently gives, for any $r\in (0,1)$,
\[\sum_{n=0}^{\infty} (r \beta)^n T^n(\mu)\leq \sum_{n=0}^{\infty}  r^n \lambda=(1-r)^{-1} \lambda.\]
Using Fubini theorem, we can now compute
\begin{eqnarray*}
T^n(\mu)(\mathbb{R}^+)&=&\int_0^{+\infty}\left( \int_{s_{n+1}}^{+\infty}\cdots \int_{s_3}^{+\infty}\left( \int_{s_2}^{+\infty} d\mu(s_1)\right)ds_2\cdots ds_{n}\right) ds_{n+1} \\
&=& \int_0^{+\infty}\left( \int_0^{s_1}\cdots \left( \int_0^{s_n} ds_{n+1}\right)\cdots ds_2\right) d\mu(s_1)\\
&=& \int_0^{+\infty}\frac{s_1^n}{n!}d\mu(s_1)
\end{eqnarray*}
so that, using monotone convergence theorem, one can obtain
\[ \int_0^{+\infty} e^{r \beta s} d\mu(s)\leq (1-r)^{-1}\lambda(\mathbb{R}^{+})\]
and our proof ends using that $(1-r)^{-1}\lambda(\mathbb{R}^{+})<\mu_0$ for sufficiently small $r$.
\par $(\Rightarrow)$ For any measurable set $\mathcal{B}$, we define
\[\lambda(\mathcal{B})=\sum_{n=0}^{\infty} \alpha^n T^n(\mu)(\mathcal{B}).\]
It is clear that $\lambda$ is a borelian measure such that $\mu \leq \lambda$. Moreover, if $\mathcal{B}$ is a measurable set, one has, thanks to monotone convergence theorem
\begin{eqnarray*}
T(\lambda)(\mathcal{B})&=&\int_\mathcal{B} \lambda([s,+\infty))ds\\
&=& \sum_{n=0}^{\infty}\alpha^n \int_\mathcal{B} T^n(\mu)([s,+\infty))ds\\
&\leq& \alpha^{-1} \sum_{n=0}^{\infty}\alpha^{n+1} T^{n+1}(\mu)(\mathcal{B}),
\end{eqnarray*}
that is, $T(\lambda)\leq \alpha^{-1} \lambda$.
\par \noindent Finally, another use of monotone convergence theorem
gives
\[\lambda(\mathbb{R}^{+}) = \int_0^{+\infty} e^{\alpha s} d\mu(s)<\mu_0.\]
\end{proof}

\begin{remark}\label{lrk1}
\begin{itemize}
\item If $\mu$ satisfies our first property, one can choose $\beta=\alpha$ in our second property.
\item If $\mu$ is supported in $(0,\tau]$, it is straightforward to see that, for some small enough constant $c$,
$$d\lambda(s)=d\mu(s)+c \ \chi_{[0,\tau]}(s) d s$$
 fulfills \eqref{mu0}. This observation allows us to recover the choices of energy in \cite{NP,NP2}.
\end{itemize}
\end{remark}

\noindent In the sequel, we can thus  consider the measure $\lambda$
obtained by the application of Proposition \ref{measure} to $|\mu|$.

\section{Well-posedness}

\subsection{General results}

Defining
\begin{equation*}
H^1_D(\Omega):= \{u\in H^1(\Omega); u=0 \text{ on } \partial \Omega_D\} \text{ and } H^1_0(\Omega):= \{u\in H^1(\Omega); u=0 \text{ on } \partial \Omega \},
\end{equation*}
we here present an application of Theorem 4.4 of Propst
and Pr\"{u}ss paper (see \cite{PP}) in the framework of hypothesis \eqref{disjoint}.

\begin{theorem}\label{Laplace}
  Suppose $u_0 \in H_D^1 (\Omega), u_1 \in L^2 (\Omega)$. Then $(S)$ admits a
  unique solution $u \in \mathcal{C} ( \mathbb{R}^+, H^1 (\Omega)) \cap
  \mathcal{C}^1 (\mathbb{R}^+, L^2 (\Omega))$ in the weak sense of Propst and Pr\"{u}ss.
  Moreover, if $u_0 \in H^2 (\Omega) \cap H^1_D (\Omega)$, $u_1 \in H^1_0
  (\Omega)$, then $u \in \mathcal{C}^1 (\mathbb{R}^+, H^1
  (\Omega)) \cap \mathcal{C}^2 (\mathbb{R}^+, L^2 (\Omega))$ and the additional results hold
  $$\forall t\geq 0,\ \Delta u(t)\in L^2(\Omega) \quad \partial_{\nu} u(t)_{|\partial \Omega_N } \in H^{1/2}(\partial \Omega_N).$$
\end{theorem}

\begin{proof}
The proof is the one proposed in \cite{PP}, Theorem 4.4 except that, for smoother data, we can not use elliptic result in the general context of \eqref{disjoint} to get more regularity.
\end{proof}

\par \noindent Inspired by \cite{NP,NP2}, we now  define the energy of the
solution of (S) at any positive time $t$ by the following formula:
\begin{eqnarray*}
  E (t) & = & \frac{1}{2} \int_{\Omega} (u' (t, x))^2 + | \nabla u (t, x) |^2
  dx + \frac{1}{2} \int_{\partial \Omega_N} m\cdot \nu \int_0^t \left( \int_0^s
  (u' (t - r, x))^2 dr \right) d \lambda (s) d \sigma\\
  & + & \frac{1}{2} \int_{\partial \Omega_N} m\cdot \nu \int_t^{\infty} \left(
  \int_0^s (u'(s - r, x))^2 dr \right) d \lambda (s) d \sigma .
\end{eqnarray*}

\begin{remark}
\begin{itemize}
 \item In the definition of energy, the measure $\lambda$ can be replaced by any positive borelian measure $\nu$
 such that $$ \nu\leq \lambda $$ such as, for instance, $|\mu|$. In fact, we will  see later that
 conditions \eqref{surmesure} and \eqref{integration} are only here to ensure that the corresponding energy $E_{\lambda}$ is non increasing, but the decay of another energy $E_{\nu}$ is implied by the decay of $E_{\lambda}$.
\item If $\mu$ is compactly supported in $[0,\tau]$, for times greater
than $\tau$, one can recover the energies from {\cite{NP,NP2}} by
choosing the   measure $\lambda$ supported in $[0,\tau]$ given by
Remark \ref{lrk1}. Indeed, the last term in the energy is null for
$t > \tau$, and the second term is reduced to
  \[ \frac{1}{2} \int_{\partial \Omega_N} m\cdot \nu \int_0^{\tau} \left( \int_0^s
     ( u'(t - r, x))^2 dr \right) d \lambda (s) d \sigma . \]
\end{itemize}
\end{remark}

We now identify our energy space.

\begin{proposition}
If $u_0 \in H^2 (\Omega) \cap H^1_D (\Omega)$, $u_1 \in H^1_{0}
  (\Omega)$, then $u'\in L^{\infty}(\mathbb{R}^+,H^1(\Omega))$.
Consequently, for such initial conditions, the energy $E(t)$ is well
defined for any $t>0$ and it uniformly  depends continuously on the
initial data.
\end{proposition}
\begin{proof}
Let us first pick some solution of $(S)$ with $u_0 \in H^2(\Omega)
\cap H_D^1 (\Omega), u_1 \in H^1_{D} (\Omega)$. We define the
standard energy as
\begin{equation*}
E_0(t)=\frac{1}{2}\int_\Omega (u'(t,x))^2+|\nabla u(t,x)|^2 d x.
\end{equation*}
As in \cite{Ko}, it is classical that
\[E_0(0)-E_0(T)=-\int_0^T\int_{\partial\Omega_N}\partial_\nu u u'd\sigma d t.\]
Using the form of our boundary condition and Young inequality, one gets, for any $\epsilon>0$,
\begin{eqnarray*}
E_0(0)-E_0(T)&=&\int_0^T\int_{\partial\Omega_N}(m\cdot\nu)\left( \mu_0 u'(t)^2+u'(t)\int_0^tu'(t-s)d\mu(s)\right) d\sigma dt\\
& \geq &  \int_0^T\int_{\partial\Omega_N}(m\cdot\nu)\left(\left(
\mu_0-\frac{\epsilon}{2}\right) u'(t)^2-\frac{1}{2 \epsilon} \left(
\int_0^tu'(t-s)d\mu(s)\right)^2\right) d\sigma
\end{eqnarray*}
Using that $\mu\leq |\mu|$ and Cauchy-Schwarz inequality consequently give us
\[E_0(0)-E_0(T)\geq \int_{\partial\Omega_N}(m\cdot\nu)\left(\left( \mu_0-\frac{\epsilon}{2}\right)
\int_0^T u'(t)^2dt-\frac{\mu_{\text{tot}}}{2 \epsilon}
\int_0^T\int_0^t(u'(t-s))^2d|\mu|(s)\right)d \sigma. \] Using now
Fubini theorem two times, one can obtain the following identities
\begin{eqnarray*}
\int_0^T\int_0^t(u'(t-s))^2d|\mu|(s)dt&=&\int_0^T \left( \int_s^T u'(t-s)^2dt\right) d |\mu|(s)\\
&=&\int_0^T \left( \int_0^{T-s} u'(t)^2dt\right) d |\mu|(s)\\
&=&\int_0^T \left( \int_0^{T-t} d |\mu|(s)\right)u'(t)^2dt \\
\end{eqnarray*}
so that, using $|\mu|([0,T-t])\leq \mu_0$,
\[E_0(0)-E_0(T)\geq \int_{\partial\Omega_N}(m\cdot\nu)\left(\left( \mu_0-\frac{\epsilon}{2}-
\frac{\mu^2_{\text{tot}}}{2 \epsilon}\right) \int_0^T
u'(t)^2dt\right) d\sigma.\] The choice of $\epsilon
=\mu_{\text{tot}}$ finally gives us that $E_0(T)$ is bounded. Using
the density of $H^2 (\Omega) \cap H^1_D (\Omega)\times  H^1_{D}
(\Omega)$ in $H^1_D (\Omega)\times  L^2 (\Omega)$, we get the
boundedness of $E_0$ for solutions with initial data $u_0 \in H_D^1
(\Omega), u_1 \in L^2 (\Omega)$. In particular, if $u_0 \in H_D^1
(\Omega), u_1 \in L^2 (\Omega)$, we obtain that $u\in
L^{\infty}(\mathbb{R}^+,H^1(\Omega))$.

Let now $u$ be a solution of $(S)$ with $u_0 \in H^2 (\Omega) \cap
H^1_D (\Omega)$, $u_1 \in H^1_0 (\Omega)$. Using Theorem
\ref{Laplace}, one can define the limit in $L^2(\Omega)$ $u_2$ of
$u''(t)$ as $t\rightarrow0$ and in this situation, as in \cite{PP},
it is easy to see that $u'$ is solution of $(S)$ with initial data
$u_1\in H^1_D(\Omega)$ and $u_2\in L^2(\Omega)$.
\par \noindent Indeed, one can see that Fubini's theorem gives
\[\int_0^tu'(t-s)d\mu(s)=\int_0^t \left( \int_0^s u''(s-r)d\mu(r)\right) ds,\]
provided $u_1=0$ on $\partial\Omega_N$, so that
\[\frac{d}{d t} \left( \int_0^tu'(t-s)d\mu(s)\right) =\int_0^t u''(t-s)d\mu(s).\]
\par \noindent Using the proof above, one concludes that $u'\in L^{\infty}(\mathbb{R}^+,H^1(\Omega))$ which, thanks to a classical trace result, give that $u'\in L^{\infty}(\mathbb{R}^+,L^2(\partial \Omega))$.
\par The first three terms of the energy $E(t)$ are consequently defined for any time $t>0$. We only need to take a look at the last one to achieve our result. Using Fubini theorem again, one has
\[\int_t^{+\infty} \left( \int_0^s (u'(t-r))^2dr\right) d \lambda(s)=\int_0^{+\infty}u'(r)^2\left( \int_{\max(r,t)}^{+\infty}d \lambda(s)\right) dr\]
so that, using \eqref{integration},
\begin{eqnarray*}
\int_{\partial \Omega_N} m\cdot \nu \int_t^{\infty} \left(
\int_0^s (u'(s - r, x))^2 dr \right) d \lambda (s) d \sigma &\leq& \|m\|_{\infty}\|u'\|_{L^{\infty}(L^2(\partial \Omega))} \int_0^{+\infty} \lambda([r,+\infty)) dr\\
&\leq& \alpha^{-1} \|m\|_{\infty} \lambda(\mathbb{R}^+)\|u'\|_{L^{\infty}(L^2(\partial \Omega))}.
\end{eqnarray*}
\end{proof}

\subsection{Compactly supported measure and semigroup approach}

In this second approach, we assume that $\mu$ is supported in $[0,\tau]$ and that
$\overline{\partial \Omega_D} \cap \overline{\partial \Omega_N} =
   \varnothing$. We here simply follow the result obtained by Nicaise-Pignotti (\cite{NP2}).

\par \noindent First, observe that, for $t > \tau$, $(S)$ is
reduced to
\[ \left\{ \begin{array}{l}
     u'' - \Delta u = 0\\
     u = 0\\
     \partial_{\nu} u + m\cdot \nu \left( \mu_0 u' (t) + \int_0^{\tau} u' (t - s)
     d \mu (s) \right) = 0\\
     u (0) = u_0\\
     u' (0) = u_1
   \end{array} \right. \left. \begin{array}{l}
     \text{in } (\tau, + \infty) \times \Omega \hspace{0.25em},\\
     \text{on } (\tau, + \infty) \times \partial \Omega_D \hspace{0.25em},\\
     \text{on } (\tau, + \infty) \times \partial \Omega_N \hspace{0.25em},\\
     \text{in } \Omega \hspace{0.25em},\\
     \text{in } \Omega \hspace{0.25em} .
   \end{array} \right. \]
We define $X_\tau=L^2(\partial \Omega_N\times (0,1)\times(0,\tau),d\sigma d\rho sd\mu(s)))$ and $Y_\tau=L^2(\partial \Omega_N\times(0,\tau);H^1(0,1),d\sigma sd\mu(s)).$
\par \noindent One can use the same strategy as in the proof of Theorem 2.1 in
{\cite{NP2}} to get

\begin{theorem}\label{Semigroup}
\begin{itemize}
\item If $u(\tau) \in H^1_D (\Omega)$, $u'(\tau) \in L^2(\Omega)$ and $u'(\tau-\rho s,x) \in X_\tau$, $(S)$ has a unique solution $u \in \mathcal{C}([\tau,+\infty),H^1_D(\Omega))\cap \mathcal{C}^1([\tau,+\infty),L^2(\Omega))$. Moreover, if $u(\tau) \in H^2(\Omega) \cap H^1_D (\Omega)$, $u'(\tau) \in H^1
(\Omega)$ and $u'(\tau-\rho s,x) \in Y_\tau$, then
\[ \left\{ \begin{array}{l}
u\in \mathcal{C}^1([\tau,+\infty),H_D^1(\Omega))\cap \mathcal{C}([\tau,+\infty),H^2(\Omega));\\
t\mapsto s u''(t-\rho s,x) \in \mathcal{C}([\tau,+\infty),X_\tau).
\end{array} \right. \]
\item If $(u_\tau^n(x),v_\tau^n(x),g^n(s,\rho,x))\rightarrow (u(\tau,x),u'(\tau,x),u'(\tau-\rho s,x))$ in $H^1_D(\Omega)\times L^2(\Omega) \times X_\tau$, then the solution $u^n$ of
\[ \left\{ \begin{array}{l}
     u'' - \Delta u = 0\\
     u = 0\\
     \partial_{\nu} u + m\cdot \nu \left( \mu_0 u' (t) + \int_0^{\tau} u' (t - s)
     d \mu (s) \right) = 0\\
     u (\tau) = u_\tau^n\\
     u' (\tau) = u_\tau^n\\
     u'(x,\tau-\rho s)=g^n(x,s,\rho)
   \end{array} \right. \left. \begin{array}{l}
     \text{in } (\tau, + \infty) \times \Omega \hspace{0.25em},\\
     \text{on } (\tau, + \infty) \times \partial \Omega_D \hspace{0.25em},\\
     \text{on } (\tau, + \infty) \times \partial \Omega_N \hspace{0.25em},\\
     \text{in } \Omega \hspace{0.25em},\\
     \text{in } \Omega \hspace{0.25em},\\
     \text{in } \Omega_N\times(0,\tau)\times (0,1) \hspace{0.25em}
   \end{array} \right. \]
is such that $E(u^n)$ converges uniformly with respect to time towards $E(u)$.
\end{itemize}
\end{theorem}

\begin{proof}
  We define $z (x, \rho, s, t) = u' (t - \rho s, x)$ for $x \in \partial
  \Omega_N$, $t > \tau$, $s \in (0, \tau)$, $\rho \in (0, 1)$.

  \noindent Problem $(S)$ is then equivalent to
  \[ \begin{array}{l}
       u'' - \Delta u = 0\\
       \\
       s z_t (x, \rho, s, t) + z_{\rho} (x, \rho, s, t) = 0\\
       \\
       u = 0\\
       \\
       \partial_{\nu} u + m\cdot \nu (\mu_0 u' (t) + \int_0^{\tau} u' (t - s) d
       \mu (s)) = 0\\
       \\
       u (\tau) = u(\tau)\\
       \\
       u' (\tau) = u(\tau)\\
       \\
       z (x, 0, t, s) = u' (t, x)\\
       \\
       z (x, \rho, \tau, s) = f_0 (x, \rho, s)
     \end{array} \begin{array}{l}
       \text{in } (\tau, + \infty) \times \Omega \hspace{0.25em},\\
       \\
       \tmop{in} \partial \Omega_N \times (0, 1) \times (0, \tau) \times
       (\tau, + \infty),\\
       \\
       \text{on } (\tau, + \infty) \times \partial \Omega_D \hspace{0.25em},\\
       \\
       \text{on } (\tau, + \infty) \times \partial \Omega_N \hspace{0.25em},\\
       \\
       \text{in } \Omega \hspace{0.25em},\\
       \\
       \text{in } \Omega \hspace{0.25em},\\
       \\
       \tmop{on} \partial \Omega_N \times (\tau, + \infty) \times (0, \tau),\\
       \\
       \tmop{on} \partial \Omega_N \times (0, 1) \times (0, \tau),
     \end{array} \]
  where $f_0 (x, \rho, s) = u' (\tau - \rho s, x)$.

  \noindent Consequently, $(S)$ can be rewritten as
  \[ \left\{ \begin{array}{l}
       U' =\mathcal{A}U\\
       U (\tau) = (u(\tau), u'(\tau), f_0)^T
     \end{array} \right. \]
  where the operator is defined by
  \[ \mathcal{A} \left( \begin{array}{l}
       u\\
       v\\
       z
     \end{array} \left) = \left( \begin{array}{c}
       v\\
       \Delta u\\
       - s^{- 1} z_\rho
     \end{array} \right) \right. \right. \]
  with domain
\begin{eqnarray*}
\mathcal{D}(\mathcal{A}) =& \{ & (u, v, z)^T\in H_D^1 (\Omega)
\times L^2 (\Omega) \times Y_\tau ;  \Delta u \in L^2 (\Omega),  \\
& & \partial_{\nu} u (x) = - (m\cdot \nu) \left( \mu_0 v (t) +
\int_0^{\tau} z (x, 1, s) d \mu (s) \right) \tmop{on} \partial
\Omega_N, v (x) = z (x, 0, s) \tmop{on} \partial \Omega_N\times(0,\tau) \ \}\ .\\
\end{eqnarray*}
The proof of Theorem 2.1 in {\cite{NP2}} shows us that $\mathcal{A}$ is a maximal monotone operator on the Hilbert space $\mathcal{H}:=H_D^1 (\Omega) \times L^2 (\Omega) \times X_\tau$ endowed with the product topology. It consequently generates a contraction semigroup on $\mathcal{H}$.
Moreover, if $(u(\tau,x),u'(\tau,x),u'(\tau-\rho s,x)) \in \mathcal{D(A)}$, one gets that
\[ \left\{ \begin{array}{l}
u\in \mathcal{C}^1([\tau,+\infty),H_D^1(\Omega))\cap \mathcal{C}([\tau,+\infty),H^2(\Omega));\\
t\mapsto s u''(t-\rho s,x) \in \mathcal{C}([\tau,+\infty),X_\tau).
\end{array} \right. \]
This ends the proof.
\end{proof}

We can consequently deduce another way to obtain solutions:

\begin{corollary}
Suppose that $u_0 \in H^2 (\Omega) \cap H^1_D (\Omega)$, $u_1 \in
H^1_0(\Omega)$, then $(S)$ has a unique solution $u \in
\mathcal{C}([\tau,+\infty),H^1_D(\Omega))\cap
\mathcal{C}^1([\tau,+\infty),L^2(\Omega))$.
\end{corollary}
\begin{proof}
Thanks to Theorem \ref{Laplace}, one only needs to check that if $u\in \mathcal{C}^1([0,\tau],H^1_D(\Omega))$ then $u'(x,\tau-\rho s)\in X_\tau$; and this is straightforward using Fubini theorem.
\end{proof}

\section{Linear stabilization}

We begin with  a classical elementary result due to Komornik
{\cite{Ko}}:

\begin{lemma}\label{komornik}
  Let $E : [0, + \infty [\rightarrow \mathbb{R}_+$ be a non-decreasing
  function that fulfils:
  \[ \forall t \geq 0 \text{, } \int_t^{\infty} E (s) ds \leq TE (t) , \]
 for some $T > 0$.
  Then, one has:
  \begin{eqnarray*}
    \forall t \geqslant T ,  E (t) \leq E (0)\exp \left(1 -
    \frac{t}{T} \right) .
  \end{eqnarray*}
\end{lemma}

We will now show the following stabilization result:

\begin{theorem}\label{tdecay}
  Assume (\ref{m})-(\ref{mu0}). Then, if $u_0 \in H^2 (\Omega) \cap H^1_D (\Omega)$, $u_1 \in
  H^1_0
(\Omega)$, there exists $T>0$ such  that the energy $E (t)$ of the
solution $u$ of $(S)$ satisfies:
  \begin{eqnarray*}
    & \forall t \geqslant T, & E (t) \leq E (0) \exp \left(1 -
    \frac{t}{T} \right) .
  \end{eqnarray*}
\end{theorem}

\begin{proof}
Our goal is to perform the multiplier method and to deal with the delay
terms to show that one can apply Lemma \ref{komornik} to the energy.

\begin{lemma}\label{Energy}
There exists $C > 0$, such that, for any solution $u$ of $(S)$ and any $S\leq T$,
\[ E(S)-E(T) \geqslant C \int_S^T\int_{\partial \Omega_N} (m\cdot \nu) \left( (u' (t))^2
+ \int^t_0 (u' (t - s))^2 d \lambda (s) \right) d \sigma dt . \]
In particular, the energy is a non-increasing function of time.
\end{lemma}
\begin{proof}
We start from the classical result that
\[ E_0(S)-E_0(T) = -\int_S^T\int_{\partial\Omega} \partial_\nu u u' d \sigma dt.\]
As above, one gets, for any $\epsilon>0$,
\[E_0(S)-E_0(T)\geq \int_S^T\int_{\partial\Omega_N}(m\cdot\nu)\left(\left( \mu_0-\frac{\epsilon}{2}\right)  u'(t)^2-\frac{\mu_{\text{tot}}}{2 \epsilon} \int_0^T\int_0^t(u'(t-s))^2d|\mu|(s)\right)d \sigma dt \]
We will now split $E-E_0$ in two terms:
$$\left[ E-E_0\right] _S^T=-\frac{1}{2}\left( \int_{\partial \Omega_N} (m\cdot\nu) [f(t,x)-g(t,x)]_S^T d \sigma\right),$$
where
\begin{eqnarray*}
f(t,x)=\int_0^t\left( \int_0^s(u'(t-r))^2dr\right) d\lambda(s),\\
g(t,x)=\int_0^t\left( \int_0^su'(r)^2dr\right) d\lambda(s).
\end{eqnarray*}
A change of variable allows us to get
\[f(t,x)=\int_0^t\int_0^t u'(r)^2drd\lambda(s)-\int_0^t\int_0^{t-s} u'(r)^2drd\lambda(s).\]
An application of Fubini theorem consequently gives us
\[f(t,x)-g(t,x)=\int_0^t u'(r)^2 \lambda([0,r])dr-\int_0^t\int_0^{t-s} u'(r)^2drd\lambda(s)\]
and, as above, one can use Fubini theorem to deduce that
\[ \int_S^T\int_0^t (u'(t-s))^2d\lambda(s)dt=\left[ \int_0^t\int_0^{t-s} u'(r)^2drd\lambda(s)\right] _S^T.\]
One now uses $\lambda ([0, r]) \leqslant \lambda(\mathbb{R}^+)$ to conclude that
$$\left[ E-E_0\right] _S^T \geq \frac{1}{2}\int_S^T\int_{\partial \Omega_N} m\cdot\nu \left( \int_0^t (u'(t-s))^2d\lambda(s)-\lambda(\mathbb{R}^+) u'(t)^2\right) d \sigma dt.$$
Summing up and using that $|\mu| \leq \lambda $, we have obtained that
\begin{eqnarray*}
E(S)-E(T)  \geqslant  \int_S^T\int_{\partial \Omega_N} (m\cdot \nu)
\left( \left( \mu_0 - \frac{\lambda(\mathbb{R}^+) + \epsilon}{2}
\right) u' (t)^2 + \frac{1}{2} \left( 1 -
\frac{\mu_{\text{tot}}}{\epsilon} \right) \int^t_0 (u' (t - s))^2 d
\lambda (s)\right)   d \sigma dt.
\end{eqnarray*}
We finally chose $\epsilon = \mu_0$ which gives us our result since $\lambda(\mathbb{R}^+) < \mu_0$.
\end{proof}

In the multiplier method, one may use Rellich's relation, especially
in the context of singularities. In our framework \eqref{disjoint},
the following Rellich inequality (see the proof of Theorem 4 in
\cite{CLO} or Proposition 4 in \cite{BLM}) is useful

\begin{proposition}\label{Rellich}
    For any $u\in H^1 (\Omega)$ such that $$\Delta u \in L^2 (\Omega), u_{| \partial \Omega_D} \in H^{\frac{3}{2}}
    (\partial \Omega_D) \text{ and } \partial_{\nu} u_{| \partial \Omega_N} \in H^{\frac{1}{2}} (\partial
    \Omega_N).$$
    Then it satisfies $2 \partial_{\nu} u (m. \nabla u) - (m\cdot \nu) | \nabla u|^2 \in
    L^1 (\partial \Omega)$ and we have the following inequality
    \begin{eqnarray*}
      2 \int_{\Omega} \triangle u (m. \nabla u) dx & \leq \int_{\Omega_{}}
      (\tmop{div} (m) I - 2 (\nabla m)^s) (\nabla u, \nabla u)^{} dx & +
      \int_{\partial \Omega} (2 \partial_{\nu} u (m. \nabla u) - (m\cdot \nu) |
      \nabla u|^2) d \sigma.
    \end{eqnarray*}
  \end{proposition}

  With this result, we can prove the following multiplier estimate:

  \begin{lemma}
    \label{IPP}Let $Mu = 2 m. \nabla u + a_0 u$, where $\left. a_0 : =
    \frac{1}{2} \left( \inf_{\bar{\Omega}} \tmop{div} (m) \right) +
    \sup_{\bar{\Omega}} \left( \tmop{div} (m) - 2 \lambda_m \right) \right)$.
    Then under the assumptions of Theorem \ref{tdecay}, the following inequality holds true:
    \[ c (m) \int_S^T \int_{\Omega} ( u')^2 + |\nabla u|^2 dxdt
       \leq - [_{} \int_{\Omega} u' Mu]^T_S \]
    \[ + \int_S^T \int_{\partial \Omega_N} Mu \partial_{\nu} u + (m\cdot \nu)
       ((u')^2 - | \nabla u|^2) d \sigma dt. \]
  \end{lemma}

  \begin{proof}
Firstly, we consider $M=2m\cdot\nabla u+a u$ where $a$ will be fixed later. Using the fact that $u$ is a regular solution of $(S)$ and noting that
    $u^{\prime \prime} Mu = (u' Mu)' - u' Mu'$, an integration by parts gives:
    \begin{eqnarray*}
      0 & = & \int_S^T \int_{\Omega} (u^{\prime \prime} - \triangle u)
      Mudxdt\\
      & = & \left[ \int_{\Omega} u' Mudx \right]_S^T - \int_S^T \int_{\Omega}
      (u' Mu' + \triangle uMu) dxdt.
    \end{eqnarray*}
    Now, thanks to Proposition \ref{Rellich}, we have :
    \begin{eqnarray*}
      \int_{\Omega} \triangle uMudx & \leq & a \int_{\Omega} \triangle uudx +
      \int_{\Omega} (\tmop{div} (m) I - 2 (\nabla m)^s) (\nabla u, \nabla
      u)^{} dx\\
      &  & + \int_{\partial \Omega} (2 \partial_{\nu} u (m. \nabla u) - (m.
      \nu) | \nabla u|^2) d \sigma .
    \end{eqnarray*}
    Consequently, Green-Riemann formula leads to:
    \[ \int_{\Omega} \triangle uMudx = \int_{\Omega} ((\tmop{div} (m) - a) I -
       2 (\nabla m)^s) (\nabla u, \nabla u)^{} dx + \int_{\partial \Omega}
       (\partial_{\nu} uMu - (m\cdot \nu) | \nabla u|^2) d \sigma . \]
    Using the fact that $\nabla u =$ $\partial_{\nu} u \nu$ on $\partial
    \Omega_D$ and $m\cdot \nu \leqslant 0$ on $\partial \Omega_D$, we have then:
    \[ \int_{\Omega} \triangle uMudx \leq \int_{\Omega} ((\tmop{div} (m) - a)
       I - 2 (\nabla m)^s) (\nabla u, \nabla u)^{} dx + \int_{\partial
       \Omega_N} (\partial_{\nu} uMu - (m\cdot \nu) | \nabla u|^2) d \sigma . \]
    On the other hand, another use of Green formula gives us:
    \[ \int_{\Omega} u' Mu' dx = \int_{\Omega} (a - \tmop{div} (m)) (u')^2 dx
       + \int_{\partial \Omega_N} (m\cdot \nu) |u' |^2 d \sigma . \]
    Consequently
    \[ \int_S^T \int_{\Omega} ( \text{div} (m) - a) (u')^2 + ((a - \text{div}
       (m)) I + 2 (\nabla m)^s) (\nabla u, \nabla u) dxdt \]
    \begin{eqnarray*}
      & \leq & \left. \left. - \left[ \int_{\Omega} u' Mudx \right]_S^T +
      \int_S^T \int_{\partial \Omega_N} \partial_{\nu} uMu + (m\cdot \nu) ((u')^2
      - | \nabla u|^2) d \sigma dt. \right. \right.
    \end{eqnarray*}
    Our goal is now to find $a$ such that $\text{div} (m) - a$ and $(a -
    \text{div} (m)) I + 2 (\nabla m)^s$ are uniformly minorized on $\Omega$.
    One has to find $a$ such that, uniformly on $\Omega$,
    \begin{equation}\label{5}
      \left\{ \begin{array}{l}
        \tmop{div} (m) - a \geq c\\
        2 \lambda_m + (a - \tmop{div} (m)) \geq c
      \end{array} \right.
    \end{equation}
    for some positive constant $c$. The latter condition is then equivalent to
    find $a$ which fulfills
    \[ \inf_{\bar{\Omega}} \tmop{div} (m) > a > \sup_{\bar{\Omega}} \left(
       \tmop{div} (m) - 2 \lambda_m \right), \]
    and its existence is now guaranteed by \eqref{m}. Moreover, it is straightforward to
    see that the greatest value of $c$ such that (\ref{5}) holds is
    \[ c (m) = \frac{1}{2} \left( \inf_{\bar{\Omega}} \tmop{div} (m) -
       \sup_{\bar{\Omega}} (\tmop{div} (m) - 2 \lambda_m) \right) \]
    and is obtained for $a = a_0$. This ends the proof.
  \end{proof}

  Consequently, the following result holds

  \begin{lemma}
    For every $\tau \leq S < T < \infty$, the following inequality holds true:
    \[ \int_S^T \int_{\Omega} ( u')^2 + | \nabla u|^2 dxdt
       \lesssim E (S) . \]
  \end{lemma}

  \begin{proof}
    We start from Lemma \ref{IPP}.

    \noindent First of all, Young~and Poincar\'e inequalities give
    \[ | \int_{\Omega} u' Mudx| \lesssim E (t), \]
    so that
    \[ \left. \left. - \left[ \int_{\Omega} u' Mudx \right]_S^T \lesssim E (S)
       + E (T) \leqslant CE (S) . \right. \right. \]
    Now, from the boundary condition, one has
    \[ Mu \partial_{\nu} u + (m\cdot \nu) ((u')_{}^2 - | \nabla u|^2) =
       (m\cdot \nu) \left( \left(\mu_0 u' + \int_0^t u' (t - s) d \mu (s) \right) Mu +
       (u')^2 - | \nabla u|^2 \right) . \]
    Using the definition of $Mu$ and Young inequality, we get for any
    $\epsilon > 0$
    \[ Mu \partial_{\nu} u + (m\cdot \nu) ((u')^2 - | \nabla u|^2)
       \leqslant (m\cdot \nu) \left( \left( 1 +\|m\|^2_{\infty} + \mu_0^2\frac{a_0^2}{2
       \epsilon} \right) (u')^2 + \frac{a_0^2}{2 \epsilon} \left(
       \int_0^t u' (t - s) d \mu (s) \right)^2 + \epsilon u^2 \right) . \]
    Another use of Poincar\'e inequality consequently allow us to choose
    $\epsilon > 0$ such that
    \[ \epsilon \int_{\partial \Omega_N} (m\cdot \nu) u^2 d \sigma \leqslant
       \frac{c (m)}{2} \int_{\Omega} | \nabla u|^2 dx. \]
    Cauchy-Schwarz inequality consequently leads to
    \[ \frac{c (m)}{2} \int_S^T \int_{\Omega} (u')^2 + | \nabla
       u|^2 dxdt \lesssim E (S) +  \int_S^T \int_{\partial \Omega_N}
       (m\cdot \nu) \left( u' (t)^2 + \int^t_0 (u' (t - s))^2 d |\mu| (s) \right) d
       \sigma dt \]
    and, since $ |\mu|\leq \lambda$, Lemma \ref{Energy} gives us the desired result:
    \[ c (m) \int_S^T \int_{\Omega} (u')^2 + | \nabla u|^2 dxdt \lesssim E (S)
       . \]
  \end{proof}

To conclude we need to absorb the two last integral terms for which
we use the following result.

\begin{lemma}

\begin{itemize}
\item For any solution $u$ and any $S < T$,
$$\int_S^T \int_{\partial \Omega_N} m\cdot \nu  \int_0^t \left( \int_0^s (
u' (t - r, x))^2 dr \right) d \lambda (s) d \sigma dt\lesssim
\int_S^T \int_{\partial \Omega_N} m\cdot \nu  \int_0^t (u' (t - s,
x))^2 d \lambda (s)d \sigma dt.$$
\item For any solution $u$ and any $S < T$,
\begin{eqnarray*}
\int_S^T \int_{\partial \Omega_N} m\cdot \nu \int_t^{+\infty} \left(
\int_0^s ( u' (s - r, x))^2 dr \right) d \lambda (s)  d \sigma
dt&\lesssim&
\int_S^T \int_{\partial \Omega_N} m\cdot \nu \int_0^t (u' (t - s, x))^2 d\lambda (s) d\sigma dt\\
&+&\int_S^{+\infty}\int_{\partial \Omega_N} m\cdot \nu \ u'^2 d
\sigma dt.
\end{eqnarray*}
\end{itemize}
\end{lemma}

\begin{proof}
\begin{itemize}
 \item
    We start from the left hand side term. We fix $x \in \partial \Omega_N, t
    \in [S, T]$ and we use Fubini theorem to estimate integrals with respect
    to time:
    \begin{eqnarray*}
     \int_0^t \left( \int_0^s (u' (t - r, x))^2 dr \right) d \lambda (s)
       &=& \int_0^t (u' (t - r, x))^2 \lambda ([r, t]) dr\\
&\leq& \int_0^t (u' (t - r, x))^2 \lambda ([r, +\infty)) dr\\
&\leq&  \alpha^{-1}\int_0^t (u' (t - r, x))^2d\lambda(r)
    \end{eqnarray*}
    which gives the required result after an integration with respect to $t$ and $x$.
\item As above, fixing $x \in \partial \Omega_N$, we obtain
\begin{eqnarray*}
     \int_S^T \int_t^{+\infty} \left( \int_0^s (u' (s - r,x)^2 dr \right) d \lambda (s)dt
       &=& \int_S^T \int_0^{+\infty} (u' (r, x))^2 \lambda ([\max(r, t),+\infty]) dr
       dt
\end{eqnarray*}
$$= \int_S^T\int_0^t (u' (t-r, x))^2 d r \lambda ([t, +\infty)) dt+\int_S^T
\left( \int_t^{+\infty} (u' (r, x))^2\lambda([r,+\infty))dr\right)
dt.$$

Since for all $r\leq t$, $\lambda ([t, +\infty) \leq \lambda ([r, +\infty))$, we first have
\begin{eqnarray*}
\int_S^T\int_0^t (u' (t-r, x))^2 d r \lambda ([t, +\infty)) dt&\leq& \int_S^T\int_0^t (u' (t-r, x))^2 \lambda ([r, +\infty))d rdt\\
&\leq& \alpha^{-1} \int_S^T\int_0^t (u' (t-r, x))^2 d\lambda(r)dt.
\end{eqnarray*}

On the other hand, Fubini theorem gives us
\[\int_S^T \left( \int_t^{+\infty} (u' (r, x))^2\lambda([r,+\infty))dr\right) dt=\int_S^{+\infty}u'(r)^2\lambda([r,+\infty)) (\min(T,r)-S)dr.\]
We now note that
$$ r\lambda([r,+\infty))\leq \int_r^{+\infty} sd\lambda(s) \leq \int_0^{+\infty} sd\lambda(s)$$
and
$$\int_0^{+\infty} sd\lambda(s)=\int_0^{+\infty}\lambda([t,+\infty))dt\leq \alpha^{-1} \lambda(\mathbb{R^+}).$$
We consequently obtain
\[\int_S^T \left( \int_t^{+\infty} (u' (r, x))^2\lambda([r,+\infty))dr\right) dt\lesssim \int_S^{+\infty}u'^2,\]
which give the required result after an integration over $\partial \Omega_N$.
\end{itemize}
  \end{proof}

Up to now, we have proven that
  \[ \int_S^T E (t) dt \lesssim E (S) + \int_S^T \int_{\partial \Omega_N} m\cdot \nu
     \int_0^t ( u' (t - s, x))^2 d \lambda (s) d \sigma dt+ \int_S^{+\infty} \int_{\partial \Omega_N} m\cdot \nu
     \ u'^2 d\sigma dt.\]
Lemma \ref{Energy} allows us to conclude since it gives
\[ \int_S^{+\infty} \int_{\partial \Omega_N} m\cdot \nu \
     u'^2 d\sigma dt\lesssim E(S)\]
and
\[\int_S^T \int_{\partial \Omega_N} m\cdot \nu
     \int_0^t ( u' (t - s, x))^2 d \lambda (s) d \sigma dt\lesssim E(S).\]
\end{proof}

\begin{remark}
In the case of some compactly supported measure $\mu$, one can also obtain exponential decay result for the following problem
\[ \left\{ \begin{array}{l}
     u'' - \Delta u = 0\\
     u = 0\\
     \partial_{\nu} u + \mu_0 u' (t) + \int_0^t u' (t - s) d \mu
     (s)= 0\\
     u (0) = u_0\\
     u' (0) = u_1
   \end{array} \right. \left. \begin{array}{l}
     \text{in } \mathbb{R}_+^{\ast} \times \mathbb{R}_+^{\ast} \hspace{0.25em},\\
     \text{on }  \mathbb{R}_+^{\ast}\times\partial \Omega_D
     \hspace{0.25em},\\
     \text{on } \mathbb{R}_+^{\ast} \times \partial \Omega_N
     \hspace{0.25em},\\
     \text{in } \Omega \hspace{0.25em},\\
     \text{in } \Omega \hspace{0.25em},
   \end{array} \right. \]
as it was done in \cite{NP2} using the work of Lasiecka-Triggiani-Yao \cite{LTY} and since the system is time invariant for $t\gg1$.
\par Moreover, a careful attention shows that our proof allows us to obtain decay for this system without assumption on the support of $\mu$ provided that
\begin{equation*}
\inf_{\partial \Omega_N} m\cdot\nu >0.
\end{equation*}
\end{remark}

\section{Examples}

We start with two general results and then particularize them to
recover results from the literature.

\begin{example}\label{ex1}
 If $\mu$ is some borelian measure such that
$$ |\mu|(\mathbb{R}^+)<\mu_0 \text{ and } \int_0^{+\infty} e^{\beta s} d |\mu|(s)< +\infty$$
for some $\beta>0$, then $\mu$ fulfils the   assumption (\ref{mu0})
for an appropriate
 $\alpha$. Indeed for any $0\leq \alpha\leq \beta$, the expression
 \[
 \int_0^{+\infty} e^{\alpha s} d |\mu|(s)
 \]
 is finite and by the dominated convergence Theorem of Lebesgue we
 have
 \[
 \int_0^{+\infty} e^{\alpha s} d |\mu|(s) \to |\mu|(\mathbb{R}^+)
 \hbox{ as } \alpha  \to 0.
 \]
Consequently by the assumption $|\mu|(\mathbb{R}^+)<\mu_0$, we get
(\ref{mu0}) for
 $\alpha$ small enough.
\end{example}

\begin{example}\label{ex2} One can choose
\[\mu=\sum_{i=1}^\infty \mu_i \delta_{\tau_i},\] where $(\tau_i)_{i=1}^\infty$, $(\mu_i)_{i=1}^\infty$ are
some families such that $\tau_i>0$ and are   two by two disjoint,
and
\begin{equation*}
\sum_{i=1}^\infty |\mu_i|e^{\alpha \tau_i}<\mu_0
\end{equation*}
for some $\alpha>0$.
\end{example}

\begin{example}\label{ex3}
 If we choose $d\mu(s)=k(s) ds$ where $k$ is  a kernel
satisfying
\begin{equation*}
\int_0^{+\infty} |k(s)| ds<\mu_0  \hbox{ and }
\int_0^{+\infty} |k(s)|e^{\beta s}ds<\infty
\end{equation*}
for some $\beta>0$. Then as a consequence of Example \ref{ex1}, we
get an exponential decay rate for the system (S) under the (very
weak) condition above, in particular we do not need any
differentiability assumptions on $k$,   nor uniform exponential
decay of $k$ at infinity as in \cite{ACS,CG,guesmia:99,NPaa}.
\end{example}

\begin{example}\label{ex4}
 Choose
 \[
d\mu(s)= k(s) \chi_{[\tau_1,\tau_2]}(s) ds, \] where
$k$ is an integrable function in $[\tau_1,\tau_2]$ such that
\[\int_{\tau_1}^{\tau_2}
|k(s)|ds<\mu_0,
\]
then we get an exponential decay for  the system (S) as a
consequence of Example \ref{ex1} because the second assumption
trivially holds. In that case we extend the results of  \cite{NP2}
to a larger class of kernels $k$,  for instance in the class of
bounded variations functions. \end{example}

\begin{example}\label{ex5}
Take
 \[
 \mu(s)= \mu_1 \delta_{\tau}(s), \]
 where $\mu_1$ is a  constant and $\tau>0$ represents
 the delay satisfying
\[|\mu_1|<\mu_0,
\]
then we recover the decay results from \cite{NP,NicValein}.
\end{example}

\end{document}